\documentclass{amsart}
\usepackage{all2018}
\usepackage{pxfonts}
\newcommand{\Sp}{\mathrm{Sp}}
\renewcommand{\phi}{\varphi}

\begin{document}

\title{The irreducible control property in matrix groups}
\author{Jan Draisma}

\begin{abstract}
This paper concerns matrix decompositions in which the factors are
restricted to lie in a closed subvariety of a matrix group. Such
decompositions are of relevance in control theory: given a target
matrix in the group, can it be decomposed as a product of elements in
the subvarieties, in a given order? And if so, what can be said about
the solution set to this problem? Can an irreducible curve of target
matrices be lifted to an irreducible curve of factorisations? We show
that under certain conditions, for a sufficiently long and complicated
such sequence, the solution set is always irreducible, and we show that
every connected matrix group has a sequence of one-parameter subgroups
that satisfies these conditions, where the sequence has length less than
1.5 times the dimension of the group.
\end{abstract}

\maketitle

\section{Two motivating examples} \label{sec:Motivating}

\subsection{$\SL_2(\CC)$ generated by two groups of shear
mappings} \label{ssec:SL2}

It is well known that every matrix $g$ in the group $\SL_2(\CC)$ of
complex $2 \times 2$-matrices with determinant $1$ can be written as
a product of matrices of the following forms: 
\[ x_1(a)=\begin{bmatrix} 1 & a \\ 0 & 1 \end{bmatrix}
\text{ and }
x_2(a)=\begin{bmatrix} 1 & 0 \\ a & 1 \end{bmatrix} \text{
for } a \in \CC, 
\]
which represent shears along the $x$-axis and along the $y$-axis,
respectively. Let $X_i:=\{x_i(a) \mid a \in \CC\}$, an additive
one-parameter subgroup of $\SL_2(\CC)$.

This paper concerns the variety of all factorisations of a target
matrix $g$ as a product of matrices in the $X_i$, in a prescribed order,
where repetitions are allowed. Since $\dim \SL_2(\CC)=3$, we need at least
three factors to reach all elements of $\SL_2(\CC)$. For $a,b,c \in \CC$ we
compute 
\begin{equation} 
\label{eq:Parm}
x_1(a) x_2(b) x_1(c)=
\begin{bmatrix}
1 + a b & a + c + a b c \\ 
b & 1 + b c 
\end{bmatrix}
=:
\begin{bmatrix}
x & y \\ z & w 
\end{bmatrix}
\end{equation}
We can recover $a,b,c$ as rational functions in $x,y,z,u$:
\[ b=z, \quad c=(w-1)/z, \text{ and } a=(x-1)/z. \]
This implies, first, that the image of the multiplication map $\mu_{121}:
X_1 \times X_2 \times X_1 \to \SL_2(\CC)$ is three-dimensional, hence
dense in $\SL_2(\CC)$; and second, that the pre-image of any
matrix $g \in \SL_2(\CC)$
with $g_{21} \neq 0$ has precisely one pre-image. Matrices with $g_{21}=0$
and $g_{11},g_{22}$ not both $1$ are {\em not} in the image of the
multiplication map. Summarising, the multiplication map is dominant
(has dense image) and birational (has generic fibres of cardinality
$1$); but it is not surjective.

This can be remedied by adding another factor: the multiplication map
$\mu_{1212}:X_1 \times X_2 \times X_1 \times X_2 \to
\SL_2(\CC)$ is surjective,
since an arbitrary element of $\SL_2(\CC)$ can be right-multiplied by some
element $x_2(-a)$ from $X_2^{-1}=X_2$---which corresponds to subtracting $a$
times the second column from the first---to make sure that the entry
at position $(2,1)$ becomes nonzero, so that the product is in $\im \mu_{121}$. 

However, the map $\mu_{1212}$ has another undesirable feature: certain fibres are
not irreducible. For instance, the solution set to the system of equations
\[ 
x_1(a) x_2(b) x_1(c) x_2(d)=
\begin{bmatrix}
1 + a b + a d + c d + a b c d & a + c + a b c \\
b + d + b c d & 1 + b c
\end{bmatrix}
= I
\] 
is the union of the two lines in $\CC^4$ with equations $b=d=a+c=0$ and
$c=a=b+d=0$. When designing a system where the control parameters
$a,b,c,d$ should vary with the target matrix $g$, it is desirable that
pre-images of irreducible varieties are irreducible themselves.

As we will see later in \S\ref{ex:SL2b}, the multiplication map
$\mu_{12121}$ still has his undesirable behaviour, but the multiplication
map $\mu_{121212}$ and those for longer words do not: for those, the
pre-image of any irreducible variety is irreducible. Note that the order
of the factors is important here; e.g., since $X_1 \cdot X_1 =X_1$,
the image of $\mu_{111222}$ is the same as that of $\mu_{12}$, and only
two-dimensional.

Our goal is to show that this behaviour is quite typical for collections of
subvarieties $(X_a)_{a \in A}$ of a matrix group $G$: under suitable
conditions, for sufficiently long and sufficiently complicated
words $w$ over the alphabet $A$, the corresponding multiplication map
$\mu_w$ has the property that the pre-image of any irreducible variety
is irreducible. It follows, for instance, that any irreducible curve
worth of matrices $g$ can be lifted to an irreducible curve worth of
factorisations.

\subsection{The ULU-decomposition} \label{ssec:ULU}

Let $L,U \subseteq \GL_n(\CC)$ be the groups of invertible
lower-triangular matrices and of upper-triangular matrices with $1$'s
on the diagonal, respectively. By the classical LU-decomposition,
the multiplication map
\[ L \times U \to \GL_n(\CC) \]
is an isomorphism of varieties with the open subset of $\GL_n(\CC)$ where
all leading principal subdeterminants are nonzero. To reach {\em all}
invertible matrices, one usually adds a factor from the finite group of
permutations matrices. Here, instead, we add another factor $U$, and
will prove the following fact. 

\begin{prop} \label{prop:ULU}
The multiplication map 
\[ \mu: U \times L \times U \to \GL_n(\CC) \]
is surjective, and, moreover, the preimage $\mu^{-1}(X)$ of every 
irreducible variety $X \subseteq GL_n(\CC)$ is irreducible.
\end{prop}

Observe that the variety on the left-hand side has dimension 
\[ n^2 + \binom{n}{2} < 1.5 \cdot \dim \GL_n(\CC). \] 
This is not a coincidence; see Theorem~\ref{thm:Any}. 

\subsection*{Organisation}
The structure of this paper is as follows. In Section~\ref{sec:Intro}
we introduce the general setting and state our main results,
Theorem~\ref{thm:Main} and Theorem~\ref{thm:Any}. We also formulate a
useful application, Proposition~\ref{prop:Curve}, about lifting curves
of matrices to curves of factorisations. In Section~\ref{sec:Proofs}
we prove Theorem~\ref{thm:Main}, Proposition~\ref{prop:ULU},
and various intermediate results of independent interest. Finally, Section~\ref{sec:Examples}
contains the worked-out example of $\SL_2(\CC)$; an example
with symplectic groups discussed in \cite{Ivarsson20}; and 
a proof of Theorem~\ref{thm:Any}.

\section{Introduction and results} \label{sec:Intro}

Let $G$ be a complex algebraic group and let $(X_a)_{a \in A}$ be a collection
of irreducible subvarieties of $G$, each containing the unit element $1
\in G$. Without loss of generality \cite[Proposition I.1.10]{Borel91},
$G$ is a closed subgroup of some $\GL_n(\CC)$ defined by the vanishing
of some polynomial equations in the $n^2$ matrix entries. But we will
not need this concrete realisation of $G$ as a matrix group.

Denote by $A^*$ the set of finite sequences (words)
in the alphabet $A$. Each $w=w_1\ldots w_l \in A^*$ gives rise to a
multiplication map
\[ \mu_w: X_w:=X_{w_1} \times \cdots \times X_{w_l} \to G,
\quad 
(x_1,\ldots,x_l) \mapsto x_1 \cdots x_l. \]
Assume that for all $a \in A$ there exists some $b \in A$ such that
$(X_a)^{-1}=X_b$, and that the $(X_a)_{a \in G}$ together generate $G$
as a group. 

\begin{de}
A word $w$ is called {\em dominant/surjective/birational} if the
map $\mu_w$ has the corresponding property. The word $w$ is called
{\em irreducible} if of all irreducible, closed subsets $Y \subseteq G$
the pre-image $\mu_w^{-1}(Y)$ is irreducible in $X_w$.
\end{de}

In this definition, $\mu_w$ and $w$ are called {\em birational} if
for $g$ in an open dense subset of $G$ the pre-image in $X_w$ consists
of a single point. 

By \cite[Proposition I.2.2]{Borel91}, surjective words exist; in
particular, since the $X_a$ are irreducible, our assumptions imply that
$G$ is a {\em connected} algebraic group (for algebraic
groups, this is equivalent to being irreducible
\cite[Proposition I.1.2]{Borel91}). 

Throughout the text, except where stated otherwise, topological terms will
refer to the Zariski topology, where the closed sets in $G$ are defined
by regular functions (restrictions of polynomials in the concrete model
of $G$ as a closed subgroup of $\GL_n(\CC)$). However, the image of $\mu_w$ is
constructible by Chevalley's theorem \cite[Corollary AG11.10.2]{Borel91},
and therefore its closure in the Zariski topology is the same as its
closure in the Euclidean topology. In particular, $\mu_w$ is
dominant in the Zariski topology if and only if it is dominant in the
Euclidean topology.

There is one notable exception to the rule that topological terms refer
to the Zariski topology, which we state now:

\begin{de}
A word $w$ is called {\em open} if the map $\mu_w$ is open in the
{\em Euclidean} topology.
\end{de}

If $\mu_w$ is open in the Euclidean topology, and if $U$ is a
Zariski-open subset of $X_w$, then $\mu_w(U)$ is Zariski-constructible
and Euclidean-open, and therefore Zariski-open. Hence $\mu_w$ is also
open in the Zariski topology, and therefore, since $G$ is
irreducible in the Zariski topology, dominant. But in our proofs
we will really use the Euclidean notion of openness: the image of any
open ball is open.

We have the following implications for a word $w
\in A^*$ and its map $\mu_w$:
\begin{center}
\includegraphics{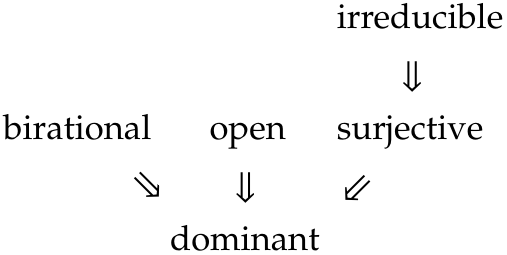}
\end{center}
For the implication irreducible $\Rightarrow$ surjective we observe that
the pre-image of any point under $\mu_w$ must be irreducible, hence in
particular non-empty.


\begin{ex} \label{ex:SL2a}
For the subgroups $X_1,X_2$ of $\SL_2(\CC)$ in
\S\ref{ssec:SL2}, the word $12$ is not dominant,
and the word
$121$ is dominant and birational. The map $\mu_{121}$ is locally open
around points $(x_1(a),x_2(b),x_1(c)) \in X_{121}$ where $b \neq 0$,
as there its derivative has full rank $3$; this is immediate
from \eqref{eq:Parm}. But it is not open around
the remaining points; e.g., no neigbourhood of $(x_1(0),x_2(0),x_1(0))$
is mapped onto a neighbourhood of $I$---indeed, no
open neighbourhood of $I$ is in the image of $\mu_{121}$. 
So $121$ is not open.  As we have
seen, the word $(1,2,1,2)$ is surjective but not irreducible. It {\em is}
open; see \S\ref{ex:SL2b}, where also longer words are discussed.
\end{ex}

We now introduce the key notion in this paper.

\begin{de}
We say that the collection $(X_a)_{a \in A}$ has the {\em irreducible
control property} if there exists a word $u$ such that each word
containing $u$ as a consecutive sub-word is irreducible.
\end{de}

For the one-parameter subgroups $X_1,X_2$ of $\SL_2(\CC)$ from
\S\ref{ssec:SL2}, the word $121212$ has the property required
of $u$ in the definition; see \S\ref{ex:SL2b}. So $(X_i)_{i=1,2}$
has the irreducible control property.

We use the term ``control'' because we can think, for each word $w \in
A^*$, of the arguments of $\mu_w:X_w \to G$ as parameters that we have
control over and that we want to tune so as to obtain a target element $g
\in G$.  The irreducible control property is clearly desirable: it tells
us that for words $w=w_1\ldots w_l$ containing $u$, the solution set
to the system of equations
\[ x_1 \cdots x_l \in Y \text{ on the control parameters } x_i \in X_{w_i},\ i=1,\ldots,l \]
is irreducible. Here is one possible application.

\begin{prop} \label{prop:Curve}
Let $w$ be an irreducible word, $g \neq g'$ elements of $G$, and $C$ an
irreducible curve in $G$ passing through $g$ and $g'$. Moreover,
let $x,x' \in X_w$ be such that $\mu_w(x)=g$ and $\mu_w(x')=g'$. Then
there exists an irreducible curve $D$ in $X_w$ passing through $x$
and $x'$ that via $\mu_w$ maps dominantly into $C$.
\end{prop}

\begin{proof}
By irreducibility of $w$, $Z:=\mu_w^{-1}(C)$ is irreducible.  Any two
points in an irreducible variety are connected by some irreducible curve
(see, e.g., \cite[page 56]{Mumford85}), hence $x,x' \in Z$ are connected
by an irreducible curve $D \subseteq Z$. Since the image of $D$ contains
two points of $C$, $\mu_w$ maps $D$ dominantly into $C$.
\end{proof}

We can now state our main theorems. 

\begin{thm}\label{thm:Main}
Let $(X_a)_{a \in A}$ be a collection of irreducible subvarieties of an
algebraic group $G$, each containing the neutral element $1 \in G$.
Assume that $\bigcup_{a \in A} X_a$ generates $G$ as a group and that
for each $a \in A$ there exists some $b \in A$ with $X_a^{-1}=X_b$. Suppose
that for some word $u=u_1\ldots u_l \in A^*$ the multiplication map
\[ \mu_u:X_u:=X_{u_1} \times \cdots \times X_{u_l} \to G,  \quad 
(x_1,\ldots,x_l) \mapsto x_1 \cdots x_l
\]
is birational. Then the collection $(X_a)_{a \in A}$ satisfies the irreducible
control property. Specifically, there exists a natural number $k$ such
that for each word $w \in A^*$ containing $k$ consecutive copies of $u$
any irreducible $Y \subseteq G$ has an irreducible pre-image
$\mu_w^{-1}(Y) \subseteq X_w$.
\end{thm}

\begin{thm} \label{thm:Any}
Every connected algebraic group $G$ has a collection $(X_a)_{a \in A}$
of connected, one-dimensional subgroups with the irreducible control
property. Indeed, there exist $n:=\dim G$ such one-parameter subgroups
$X_1,\ldots,X_n$ and a word $u$ of length $< 1.5 \cdot \dim G$ ($=$ if $G$ is
the trivial group) such that each word containing $u$ as a consecutive
sub-word is irreducible.
\end{thm}

\begin{re}
Theorem~\ref{thm:Any} can be interpreted as follows: if we want to use
one-parameter subgroups in designing a multiplicative system that can
reach all elements of the group $G$, then for dimension reasons we need
at least $n$ of these groups to reach all elements of $G$. At the cost
of choosing (less than) 1.5 times as many, we can ensure that irreducible
varieties lift to irreducible factorisation varieties. For $\GL_n(\CC)$,
the ULU-decomposition from Section~\ref{ssec:ULU} is of this form,
if we write $U$ and $L$ as suitable products of one-parameter groups.

We do not know if the factor $1.5$ is optimal---it is conceivable, for
instance, that for a different choice of one-parameter subgroups of $G$,
a word of length $n$ suffices.
\end{re}

\begin{re}
We will only use classical facts from the huge literature on matrix
decompositions, such as the LU decomposition and its generalisation, the
Bruhat decomposition. Nevertheless, we would like to
point out one recent paper that, although it concerns matrix decompositions of
a different nature than ours, uses techniques from algebraic groups
similar to our techniques: in \cite{Ye16}, it is proved that every matrix
is a product of Toeplitz matrices and also a product of
Hankel matrices, and a bound on the number of factors is
given. It would be interesting to see whether the factorisation spaces 
are also irreducible. A complicating factor there, 
is that the matrices are not required to be invertible, like they are
here.
\end{re}

\section{Proofs} \label{sec:Proofs}

In this section, we prove Theorem~\ref{thm:Main} and
Proposition~\ref{prop:ULU}. We retain
the notation from Section~\ref{sec:Intro}. 

\begin{lm} \label{lm:Subword}
If a word $u$ is dominant/surjective/open, then any word
$w$ containing $u$ as a consecutive sub-word has the same
property.
\end{lm}

\begin{proof}
Since $1 \in X_a$ for each $a \in A$, we have $\im \mu_w \supseteq \im
\mu_u$, so dominance or surjectivity of $u$ implies that of $w$ (here we
do not even need that the letters of $u$ appear at consecutive positions
in $w$). For openness, write $w$ as a concatenation $w_1 u w_2$ and let
$(x,y,z) \in X_{w_1} \times X_u \times X_{w_2} = X_w$. Let $U$ be an open
neighbourhood of $(x,y,z)$ in the Euclidean topology. The intersection
of $U$ with $\{x\} \times X_{w_1} \times \{z\}$ is of the form $\{x\}
\times V \times \{z\}$ with $V \subseteq X_{w_2}$ open in the Euclidean
topology. By openness of $u$, $\mu_u(V)$ contains an open neighbourhood
$O$ of $\mu_u(y)$ in the Euclidean topology. But then $\mu_{w_1}(x)
\cdot O \cdot \mu_{w_2}(z)$ is an open neighbourhood of $\mu_w(x,y,z)$
contained in $\mu_w(U)$.
\end{proof}

In our examples in Section~\ref{sec:Examples}, the $X_a$ will be connected
subgroups of $G$. By the following lemma, we may then restrict to
words without consecutive repeated letters.

\begin{lm} \label{lm:Repetition}
Suppose that each $X_a$ is a closed, connected subgroup of $G$. Let $w
\in A^*$ and let $u \in A^*$ be obtained from $w$ by replacing every run
of consecutive copies of any letter $b$ by a single $b$. Then $w \in A^*$
is dominant/surjective/open/irreducible if and only if $u$
has the corresponding property.
\end{lm}

\begin{proof}
It suffices to prove the result when $w=w_1bbw_2$ and $u=w_1bw_2$. Since
$X_b \cdot X_b = X_b$ we have $\im \mu_w=\im \mu_u$ and hence $w$ is
dominant/surjective iff $u$ is. Furthermore, consider the multiplication map 
\[ 
\phi: 
X_w = X_{w_1} \times X_b \times X_b \times X_{w_2} \to X_{w_1} \times
X_b \times X_{w_2}=X_u, \quad (x,s,s',z) \mapsto (x,ss',z).
\]
We have $\mu_w = \mu_u \circ \phi$ and $\phi$ is open, so if $u$ is open,
then so is $w$. On the other hand, for any Euclidean-open $O \subseteq X_u$ we have
$\mu_u(O)=\mu_w(\phi^{-1}(O))$, so if $w$ is open, then so is $u$. Now
let $Y \subseteq G$ be irreducible. Then $\mu_u^{-1}(Y)=\phi(\mu_w^{-1}(Y))$,
so if $w$ is irreducible, then so is $u$.  Conversely, $\mu_w^{-1}(Y)$
is the image of $X_b \times \mu_u^{-1}(Y)$ under the map $((s),(x,s',z))
\mapsto (x,s's,s^{-1},z)$, so since $X_b$ is irreducible, if $u$ is
irreducible, then so is $w$.
\end{proof}

\begin{lm} \label{lm:DomSur}
Dominant and surjective words exist.
\end{lm}

This is well known; we recall the argument from \cite[Proposition
I.2.2]{Borel91}. 

\begin{proof}
Let $w$ be a word such that $H:=\overline{\im \mu_w}$ has maximal
dimension. Then $X_a H \subseteq H$ for all $a \in A$ and
since $\bigcup_{a \in A} X_a$ is closed under inversion and 
generates $G$ we have $H=G$. Hence $w$ is dominant. By Chevalley's theorem,
$\im \mu_w$ contains an open, dense subset $U$ of $G$. Then for each $g
\in G$ the set $U \cap U^{-1} g$ is nonempty, so that there exist $h,h'
\in U$ with $hh'=g$. So $UU=G$ and therefore the concatanation $ww$
is surjective.
\end{proof}

As remarked before, any irreducible word is also surjective. But
unlike surjective words, irreducible words need not exist;
see the following example. 

\begin{ex} \label{ex:Torus}
Let $G=(\CC^*)^2$ and let $X_1=\{(t,t^2) \mid t \in \CC^*\}$ and
$X_2=\{(t^2,t) \mid t \in \CC^*\}$. Since $X_1$ and $X_2$ are subgroups,
by Lemma~\ref{lm:Repetition}
we may restrict our attention to words in which the letters $1,2$
alternate. For definiteness, consider $w=1212$. Then $\mu_w$
is the homomorphism of tori
\[ \mu_w: X_1 \times X_2 \times X_1 \times X_2 \cong
(\CC^*)^4 \to (\CC^*)^2,\quad  (t_1,t_2,t_3,t_4) \mapsto (t_1 t_2^2
t_3t_4^2,t_1^2t_2t_3^2t_4).
\]
Here the last map is the momomial map whose exponent
vectors are the rows of the $2 \times 4$-matrix
\[ 
\begin{bmatrix}
1 & 2 & 1 & 2 \\
2 & 1 & 2 & 1
\end{bmatrix}
\]
Let $b_1,b_2 \in \ZZ^4$ be the rows of this matrix. The fact that
$b_1,b_2$ do not span a saturated lattice---e.g.  $b_1/3+b_2/3 \in
\ZZ^4$---implies that $\ker \mu_w=\mu_w^{-1}(1,1)$ is not irreducible (it has $3$
irreducible components). Hence $w$ is not irreducible. The same applies
to longer words. 
\hfill $\clubsuit$
\end{ex}

An important difference between Example~\ref{ex:Torus} (where the
irreducible control property does not hold) and the example
of $\SL_2(\CC)$ in 
\S\ref{ex:SL2b} 
(where it {\em does} hold) is the existence of birational words in the
latter case and the non-existence of birational words in the former
case. This explains the condition in Theorem~\ref{thm:Main}.

We now set out to prove the existence of {\em open} words. A well-known
sufficient condition for a map to be open is that its derivative
is surjective at every point. This requirement will be too
restrictive for our purposes. For instance, in the example of $\SL_2(\CC)$
from \S\ref{ssec:SL2}, for any word $w=w_1 \ldots w_l$
that contains both letters $1,2$ at least once, the derivative at
$(x_{w_1}(0),\ldots,x_{w_l}(0))=(I,\ldots,I)$ of the multiplication map
has rank $2$ rather than $3$: its image is spanned by all matrices in
the Lie algebra $\liea{sl}_2$ with zeroes on the diagonal. 

We will show, however, that the set of such bad points is small enough
in a suitable sense. To this end, for any word $u \in A^*$, we define
\[ J_u:=\{x \in X_u \mid d_x \mu_u \text{ is not surjective} \}; \]
here $d_x \mu_u$ is the derivative $T_x X_u \to T_{\mu_u(x)} G$ at $x$
of the map $\mu_u$. If $u$ is not dominant, then $J_u$ is all of $X_u$;
otherwise, $\dim J_u \leq \dim X_u-1$.

\begin{lm} \label{lm:Jac}
For any two words $u,w \in A^*$ we have $J_{uw} \subseteq J_u \times J_w$.
In particular, writing $u^n$ for the concatenation of $n$ copies of a
{\em dominant} word $u$, we have $\dim J_{u^n} \leq n (\dim X_u - 1)$.
\end{lm}

\begin{proof}
For $(x,y) \in X_u \times X_w=X_{uw}$ we have 
\[ \im d_{(x,y)} \mu_{uw} \supseteq \mu_u(x) \cdot (\im d_y \mu_w) + (\im
	d_x \mu_u) \cdot \mu_w(y) \quad \text{(in $
	T_{\mu_u(x)\mu_w(y)} G$),} \]
so if the left-hand side has dimension less than $\dim G$, then also
$d_x \mu_u, d_y \mu_w$ have rank less than $\dim G$.  The last
statement follows from $\dim J_u \leq \dim X_u - 1$.
\end{proof}

\begin{prop} \label{prop:Open}
Open words exist. More specifically, if $u$ is a dominant
word, then $u^{\dim G + 1}$ is open.
\end{prop}

\begin{proof}
Let $u$ be a dominant word and set $d:=\dim X_u$. For a positive integer
$n$ consider $w:=u^n$. Fix a point $x \in X_w$ and consider an irreducible
component $F \ni x$ of the fibre $\mu_w^{-1}(\mu_w(x))$. Now $\dim F \geq
nd-\dim G$ by properties of fibre dimension \cite[\S8]{Mumford88}, while
$\dim J_w \leq n(d-1)$ by Lemma~\ref{lm:Jac}. Hence if $n>\dim G$, then
$F$ is not contained in $J_w$. Hence any open ball $B$ around $x$ (in the
Euclidean topology) contains a $y \in F \setminus J_w$. Since $d_y \mu_w$
is surjective, $\mu_w$ maps an open ball in $B$ around $y$ onto an open
neighbourhood (in the Euclidean topology) of $\mu_w(y)=\mu_w(x)$. Hence
$\mu_w$ is open in the Euclidean topology.
\end{proof}

\begin{prop} \label{prop:Irred}
If $w$ is an open word and $u$ is a birational word, then for
any $s,t \in A^*$ the concatenation $swut$ is irreducible. 
\end{prop}

\begin{proof}
Let $Y \subseteq G$ be an irreducible, closed subset, and set
$F:=\mu_{swut}^{-1}(Y)$. We will show that $F$ is irreducible.

There exists an open, dense subset $O$ of $X_u$ such that $\mu_u$
restricts to an isomorphism from $O$ to an open, dense subset $P$ of
$G$. Let $\phi:P \to O$ be the inverse of that isomorphism. 

If $\mu_{swut}$ maps $(x_s,x_w,x_u,x_t) \in X_{swut}$ to $y
\in Y$, then 
\[ \mu_u(x_u)=\mu_w(x_w)^{-1} \mu_s(x_s)^{-1} y \mu_t(x_t)^{-1} \]
and therefore, if $x_u \in O$, then we have 
\[ x_u=\phi(\mu_w(x_w)^{-1} \mu_s(x_s)^{-1} y \mu_t(x_t)^{-1}). \]
Let $Q$ be the open subset of $X_s \times X_w \times X_t
\times Y$ defined by 
\[ Q:=\{(x_s,x_w,x_t,y) \mid 
	\mu_w(x_w)^{-1} \mu_s(x_s)^{-1} y \mu_t(x_t)^{-1} \in P\}. \]
Since $w$ is open and {\em a fortiori} dominant, $Q$ is nonempty, hence dense in $X_s
\times X_w \times X_t \times Y$, hence irreducible. Define
the morphism $\psi: Q \to F$ by 
\[ \psi(x_s,x_w,x_t,y) = (x_s,x_w,
\phi(\mu_w(x_w)^{-1} \mu_s(x_s)^{-1} y
\mu_t(x_t)^{-1}),x_t).  \]
The image of $\psi$ is an irreducible subset of $F$ that contains all points
$(x_s,x_w,x_u,x_t) \in F$ for which $x_u$ lies in $O$. We
claim that $F=\overline{\im \psi}$, so that $F$ is, indeed,
irreducible. 

For this it suffices to prove that for any $(z_s,z_w,z_u,z_t) \in F$ and any
open neighbourhood $\Omega$ of $(z_s,z_w,z_u,z_t)$ in $X_{swut}$ there exists a
point $(x_s,x_w,x_u,x_t) \in F \cap \Omega$ with $x_u \in O$. We can in
fact take $x_s:=z_s$ and $x_t:=z_t$ and only vary $z_w$ and $z_u$. Indeed,
the neighbourhood $\Omega$ contains $\{z_s\} \times B_w \times B_u \times
\{z_t\}$ for small balls $B_w$ and $B_u$ around $z_w \in X_w$ and $z_u
\in X_u$, respectively. As $w$ is open, $\mu_w(B_w)$ contains an open
ball $B'_w$ around $\mu_w(z_w) \in G$. If we take any $x_u$ in $B_u$
sufficiently close to $z_u$, then $\mu_w(z_w) \mu_u(z_u) \mu_u(x_u)^{-1}
\in B'_w$ and hence there exists an $x_w \in B_w$ such that $\mu_w(x_w)
\mu_u(x_u)=\mu_w(z_w) \mu_u(z_u)$. Since $O$ is dense in
$X_u$, we may take such an $x_u \in O \cap B_u$ and have
thus found a point $(z_s,x_w,x_u,z_t) \in F \cap \Omega$
with $x_u \in O$.
\end{proof}

\begin{proof}[Proof of Theorem~\ref{thm:Main}.]
By assumption, a birational word $u \in A^*$ exists. By
Proposition~\ref{prop:Open}, an open word $w \in A^*$ exists;
indeed, some concatenation of $u^n$ of copies of $u$ is open.
By Proposition~\ref{prop:Irred}, any word in $A^*$ containing $u^{n+1}$
as a consecutive sub-word is irreducible. Hence $(X_a)_{a \in A}$ has
the irreducible control property.
\end{proof}

\begin{proof}[Proof of Proposition~\ref{prop:ULU}.]
Set $X_1:=L$ and $X_2:=U$. By the classical LU-decomposition, the
word $u:=12$ is open and birational, and so ist the word
$w:=21$ (the transpose of the LU-decomposition is the
UL-decomposition). 
By Proposition~\ref{prop:Irred}, any word containing $wu=2112$ is
irreducible. Finally, by Lemma~\ref{lm:Repetition}, we may replace the
two consecutive $1$s by a single $1$, i.e., every word containing $121$
is irreducible. This proves Proposition~\ref{prop:ULU}.
\end{proof}

\section{Examples} \label{sec:Examples}

\subsection{The case of $\SL_2(\CC)$}  \label{ex:SL2b}

Recall the subgroups $X_1,X_2$ of $\SL_2(\CC)$ from
\S\ref{ssec:SL2}. By Lemma~\ref{lm:Repetition},
we need only look at words where the letters $1,2$ alternate. In
\S\ref{ssec:SL2} we already saw that $121$ is dominant
and $1212$ is surjective but not irreducible. In Example~\ref{ex:SL2a}
we saw, moreover, that $121$ is not open.

We claim that $1212$ {\em is} open. By the analysis in
Example~\ref{ex:SL2a}, it is certainly locally open around points
$(x_1(a),x_2(b),x_1(c),x_2(d))$ with $b \neq 0$ or, similarly, $c \neq
0$. Moreover, by acting with $x_1(-a)$ from the left and $x_2(-d)$ from
the right, we see that it suffices to check local opennes at the point
where $a=b=c=d=0$. Suppose, then, that we want to solve
\begin{align*}
&x_1(a) x_2(b) x_1(c) x_2(d)=
\begin{bmatrix}
1 + a b + a d + c d + a b c d & a(1+bc) + c \\
d(1+bc) +b & 1 + b c
\end{bmatrix}
=\\
&=
\begin{bmatrix}
x & y \\ z & w
\end{bmatrix} \in \SL_2(\CC),
\end{align*}
where $x,w \approx 1$ and $y,z \approx 0$. Then we can find small
$b,c$ such that $1+bc=w$ (e.g., both equal to the same
square root of $w-1$), and after that $a$ and $d$ determined
by $a=(y-c)/(1+bc)$ and $d=(z-b)/(1+bc)$ are also small. The last
condition $1+ab+ad+cd+abcd=x$ now follows, since the right-hand
matrix lies in $\SL_2(\CC)$. This shows that $\mu_{1212}$ is open near
$(x_1(0),x_2(0),x_1(0),x_2(0))$, so $1212$ is an open word.

Next, $12121$ inherits the surjectivity and openness
from $1212$ by Lemma~\ref{lm:Subword}, but is still not
irreducible. For example, the fibre $\mu_{12121}^{-1}(I)$
consists of all quintuples 
$(\phi_1(a),\phi_2(b),\phi_1(c),\phi_2(d),\phi_1(e))$ with either
$c=0=b+d=a+e$ or $d=0=b=a+c+e$.

Finally, we claim that $121212$ and all larger alternating words are
irreducible. Indeed, the word $w=1212$ is open, the word $u=212$
is birational (by an argument similar to that for $121$), and hence
$swut=s1212212t$ is irreducible by Proposition~\ref{prop:Irred} for all
words $s,t$. Now apply Lemma~\ref{lm:Repetition} to replace $22$ by $2$.

\subsection{On a question by Kutzschebauch}

This paragraph concerns an example communicated to me by Frank
Kutzschebauch; see \cite{Ivarsson20}.
Let $G:=\Sp_{2n}(\CC)$, the complex symplectic group preserving the symplectic
form $\langle (b,c),(d,e) \rangle :=  b e^T - c d^T$, where $b,c,d,e \in
\CC^n$ and $(b,c),(d,e) \in \CC^{2n}$ are thought of as row vectors. 

Let
\[ X_1:=\left\{ \begin{bmatrix} 1 & 0 \\ A & 1 \end{bmatrix}
\mid A^T=A\right\} \text{ and }
X_2:=\left\{ \begin{bmatrix} 1 & B \\ 0 & 1 \end{bmatrix}
\mid B^T=B\right\},\]
two subgroups of $G$ isomorphic, as algebraic groups, to the
vector space $\CC^{\binom{n+1}{2}}$.

\begin{thm} \label{thm:Sp2n}
Let $w$ be a word in the alphabet $\{1,2\}$ in which the letters $1$ and
$2$ alternate, and assume that $w$ is sufficiently long.  Let $(b,c),(d,e)
\in \CC^{2n}$. Then the closed set
\[ \{x \in X_w \mid (b,c) \mu_w(x)=(d,e) \} \]
is irreducible. 
\end{thm}

The following lemmas are proved by straightforward calculations.

\begin{lm}
The map $\mu_{121}$ maps $X_{121}$ birationally to the closed subset 
\[ Z:=\left\{ \begin{bmatrix} C & D \\ E & F \end{bmatrix}
\mid D^T=D \right\} \subseteq G. \]
\ \hfill $\square$
\end{lm}

Let $S$ be a sufficiently general codimension-$n$ subspace of the space
of symmetric $n \times n$-matrices, let $T$ be a vector
space complement of $S$ (of dimension $n$) and define 
\[ X_3:=\left\{ \begin{bmatrix} 1 & B \\ 0 & 1 \end{bmatrix}
\mid B \in S\right\}, \quad 
X_4:=\left\{ \begin{bmatrix} 1 & B \\ 0 & 1 \end{bmatrix}
\mid B \in T\right\} \subseteq X_2, \]
so that the multiplication map $X_3 \times X_4 \to X_2$ is an isomorphism.

\begin{lm}  \label{lm:1213}
The map $\mu_{1213}$ maps $X_{1213}$ birationally onto $G$.
\end{lm}

In particular, by Theorem~\ref{thm:Main}, the collection $X_1,X_2,X_3$
has the irreducible control property. We now use the
ingredients for that theorem to prove
Theorem~\ref{thm:Sp2n}.

\begin{proof}[Proof of Theorem~\ref{thm:Sp2n}.]
By the previous lemma and $X_2 \supseteq X_3$, the word $v:=1212$ is
dominant. Hence by Proposition~\ref{prop:Open}, $w:=v^{\dim G + 1}$
is open. Hence by Lemma~\ref{lm:1213} and Proposition~\ref{prop:Irred},
any word containing $w1213$ is irreducible. In particular, this holds
for any word containing $w12134$.  Since $X_3 \times X_4 \to X_2$ is an
isomorphism, it holds for any word $u$ containing $w1212=v^{\dim
G + 2}$.

Finally, for $Y$ choose the set $\{g \in G \mid (b,c)g=(d,e)\}$,
which is a coset of the irreducible stabiliser of $(b,c)$ in
$G$, hence irreducible. Therefore $\mu_u^{-1}(Y)$ is
irreducible, as desired.
\end{proof}

\begin{re}
Using direct computations it is likely possible to find much shorter
open words, so that ``sufficiently long'' in Theorem~\ref{thm:Sp2n}
becomes a much milder condition. Indeed, for $n=2$ and $(b,c)=(0,0,0,1)$
and alternating words starting with $1$, \cite[Lemma
10.5]{Ivarsson20}
states that Theorem~\ref{thm:Sp2n} holds for words of length at least
$5$ for all values of $(d,e)$, while for the word $1212$ certain vectors
$(d,e)$ have reducible fibres.
\end{re}

\subsection{The case of general algebraic groups}

Let $G$ be a connected algebraic group. We will prove
Theorem~\ref{thm:Any}.

\begin{proof}[Proof of Theorem~\ref{thm:Any}.]
Let $R$ be the unipotent radical of $G$, and let $L$ be a Levi complement
of $R$, i.e., a subgroup of $G$ such that the map $L \times R \to G$ is
an isomorphism of varieties, so that $G$ is isomorphic to a semi-direct
product $L \ltimes R$; such a group exists by \cite{Mostow56}. Then $L$
is a reductive group, and admits a decomposition akin to the
LU decomposition in $\GL_n(\CC)$; the details are as
follows.

Let $B_+$ be a Borel subgroup of $L$, $T$ a maximal torus in $B_+$,
$U_+$ the unipotent radical of $B_+$, and $U_-$ the unipotent radical
of the Borel group $B_-$ opposite to $B_+$, i.e. such that $B_- \cap B_+
= T$. By the Bruhat decomposition, the multiplication map $U_- \times T
\times U_+ \to L$ is an isomorphism of varieties with an open subvariety
of $L$ by \cite[Theorem IV.14.12]{Borel91}. (For
$L=\GL_n(\CC)$, this is
just the LU-decomposition seen in \S\ref{ssec:ULU}.) In particular,
this map is open and birational. 

Consequently, also the multiplication map 
\[ U_- \times T \times U_+ \times R \to G \]
is open and birational. Similarly, so is the map 
\[ R \times U_+ \times T \times U_- \to G. \]
Now $\tilde{B}_+:=T \cdot U_+ \cdot R = R \cdot U_+ \cdot T$ is a subgroup of $G$---in
fact, a Borel subgroup of $G$---and the multiplication maps $T \times U_+
\times R \to \tilde{B}_+$ and $R \times U_+ \times T \to \tilde{B}_+$
are isomorphisms of varieties. It follows that the multiplication maps
\[ U_- \times \tilde{B} \to G \text{ and } \tilde{B} \times
U_- \to G \]
are both birational and open. Hence, using
Proposition~\ref{prop:Irred} as in the proof of
Proposition~\ref{prop:ULU}, all pre-images of irreducible
varieties in $G$ under the multiplication map 
\[ U_- \times \tilde{B} \times \tilde{B} \times U_- \to G \]
are irreducible. Now use Lemma~\ref{lm:Repetition} to
conclude that the multiplication map 
\begin{equation} \label{eq:General} U_- \times \tilde{B} \times U_- \to G 
\end{equation}
has the same property. 

To find the one-parameter subgroups, we proceed as follows. 
Set $l:=\dim U_-=\dim U_+$ and $m:=\dim T$ and $k:=\dim R$.

We have $T \cong (\CC^*)^m$, and this yields $m$ isomorphic copies
$X_{l+1},\ldots,X_{l+m} \subseteq T$ of $\CC^*$ such that the
multiplication map $X_{l+1} \times \cdots \times X_{l+m} \to T$ is an
isomorphism of varieties (and even of algebraic groups).

Furthermore, for any connected unipotent algebraic group $H$, there
exists a basis $v_1,\ldots,v_p$ of the Lie algebra $\liea{h}$ of $H$
such that the one-parameter subgroups $H_i:=\exp(\CC v_p)$
(which are algebraic subgroups!) have the
property that the product map $H_1 \times \cdots \times H_p \to H$
is an isomorphism of varieties.

Applying the previous paragraph to the groups $U_-,U_+,R$ of dimensions
$l,l,k$, and combining these with the one-parameter subgroups
$X_{l+1},\ldots,X_{l+m}$ of $T$, we find one-parameter subgroups such
that the composition of the multiplication maps:
\begin{align*} 
&(X_1 \times \cdots \times X_l) \times (X_{l+1} \times
\cdots \times X_{l+m}) \times (X_{l+m+1} \times \cdots \times
X_{2l+m}) \times\\ &\times (X_{2l+m+1} \times \cdots \times
X_{2l+m+k}) \times
(X_1 \times \cdots \times X_l) \to 
U_- \times T \times U_+ \times R \times U_- \to G
\end{align*}
has the property that all preimages of irreducible varieties
in $G$ are irreducible, and that, indeed, all words in the alphabet
$\{1,\ldots,2l+m+k\}$ containing the word
$(1,2,\ldots,2l+m+k,1,2,\ldots,l)$ are irreducible. Finally,
we observe that the word above has length 
\[ 3l+m+k < (3/2) (2l + m + k) = 1.5 \cdot \dim G. \]
\end{proof}

\begin{re}
The proof yields a slightly better bound than $1.5 \cdot \dim G$,
namely, $\dim G + (\dim L)/2$, where $L$ is the levi complement of $G$,
or even $\dim G + (\dim L - l)/2$, where $l$ is the {\em rank} of $G$
(and of $L$), namely, the dimension of a maximal torus.  However, if $G$
is restricted to the classical groups (roughly speaking, the
general linear groups, the orthogonal groups, and the
symplectic groups), so that $L=G$, then $l$ is proportional 
to the square root of $\dim G$, so that this bound is not much better
than $1.5 \cdot \dim G$. It would be interesting to show the existence
of irreducible words of shorter length, say around $\dim G$.
\end{re}

\begin{ex}
For $\SL_2(\CC)$, if in addition to $X_1,X_2$ from
\S\ref{ssec:SL2} we take the multiplicative
one-parameter subgroup 
\[ X_3:=\left\{\begin{bmatrix} t & 0 \\ 0 & t^{-1}
\end{bmatrix} \mid t \in \CC^* \right\}, \]
then the proof above says that the word $2312$, and all words containing
it, are irreducible.
\end{ex}

\bibliographystyle{alpha}
\bibliography{diffeq}

\end{document}